\documentclass[11pt]{amsart}

\usepackage[usenames]{color}

\usepackage{amsmath}
\usepackage{amsthm}
\usepackage{amsfonts}
\usepackage{amssymb}
\usepackage{latexsym}
\usepackage{color}
\usepackage{graphicx}
\usepackage{float}

\usepackage{blindtext}

\usepackage{hyperref}
\usepackage[top=1in, bottom=1.5in, left=1in, right=1in]{geometry}

\newtheorem{theorem}{Theorem}

\newtheorem{conjecture}{Conjecture}

\newtheorem{problem}{Problem}

\newtheorem{proposition}{Proposition}

\newcommand{\Z}{\mathbb{Z}}

\newcommand{\Q}{\mathbb{Q}}

\newcommand{\R}{\mathbb{R}}

\begin{document}

\author[]{Vladimir Shpilrain}
\address{Department of Mathematics, The City College of New York, New York,
NY 10031} \email{shpilrain@yahoo.com} 

\title[Average-case complexity of the Whitehead problem]{Average-case complexity of the Whitehead problem\\
for free groups}

\begin{abstract}
The worst-case complexity of group-theoretic algorithms has been studied for a long time. Generic-case complexity, or complexity on random inputs, was introduced and studied relatively recently. In this paper, we address the average-case complexity (i.e., the expected runtime) of algorithms that solve a well-known problem, the Whitehead problem in a free group, which is: given two elements of a free group, find out whether there is an automorphism that takes one element to the other. First we address a special case of the Whitehead problem, namely deciding if a given element of a free group is part of a free basis. We show that there is an algorithm that, on a cyclically reduced input word, solves this problem and has constant (with respect to the length of the input) average-case complexity.  For the general Whitehead problem, we show that the classical Whitehead algorithm has linear average-case complexity if the rank of the free group is 2. We argue that the same should be true in a free group of any rank but point out obstacles to establishing this general result.

\end{abstract}


\maketitle

\section{Introduction}

The worst-case complexity of group-theoretic algorithms has been studied probably since small cancellation groups were introduced and it was noticed that the word problem in these groups admits a linear time (with respect to the ``length" of an input) solution, see e.g. \cite{L-S}.

Genericity of group-theoretic properties and generic-case complexity of group-theoretic algorithms were  introduced more recently, see \cite{AO} and \cite{generic}. Average-case complexity of the word and subgroup membership problems in some (infinite) groups was addressed in \cite{average}.

Studying average-case (as well as generic-case) complexity requires introducing some kind of measure on the group in question. One of the most natural ways of doing that is using {\it asymptotic density}. Let $G$ be a given group, $U \subseteq G$ a subset of $G$, and let $B_n$ be the ball of radius $n$ in the Cayley graph of $G$, i.e., the set of elements of $G$ that can be represented by words of length $\le n$ in the generators of $G$.
Then the {\it asymptotic density} of $U$ is $\limsup_{n \to \infty} \frac{|U \cap B_n|}{|B_n|}$.
We note, in passing, that this definition depends on a set of generators of the group $G$, so a generating set has to be fixed up front.

A set $U \subseteq G$ is called {\it negligible} in $G$ if its asymptotic density is 0 and {\it generic} in $G$ if its asymptotic density is 1. If convergence to 0 (respectively, to 1) is exponentially fast, then the set $U$ is called  {\it exponentially negligible} (respectively, {\it exponentially generic}). For some ``smooth" groups (like a free group, for example), one can use the {\it sphere} $S_n$ in the Cayley graph of $G$, instead of the ball, in the above definitions; this usually does not change the property of being negligible or generic but makes computations easier. Depending on a particular algorithmic problem at hand, it may be more natural to use other stratifications of the Cayley graph; for example, in \cite{Rivin}, {\it annuli} were used instead of spheres or balls, and the corresponding {\it annular density} was used in place of asymptotic density as defined above.

In this paper, we will use stratification by the spheres $S_n$ in the Cayley graph of a free group (with respect to a fixed generating set). Having defined the measure on a free group accordingly, we can now address the average-case time complexity (= the expected runtime) of algorithms whose inputs are elements of a free group, i.e., freely reduced words in the generators and their inverses.

A particular algorithmic problem that we focus on in this paper is the {\it Whitehead problem} for a free group.
The problem is: given two elements, $u$ and $v$, of a free group $F$, find out whether there is an automorphism of $F$ that takes $u$ to $v$. The generic-case complexity of the {\it Whitehead algorithm} for solving this problem is known to be linear \cite{White}, but the worst-case complexity is unknown. It is not even known whether it is bounded by a polynomial in $\max(|u|, |v|)$ or not  (cf. \cite{Problems}, Problem (F25)).

In Section \ref{orbits} of this paper, we treat a special case of this problem where $v$ is a free generator of $F$. What makes this case interesting is that, as it turns out, there is an algorithm for solving this problem whose expected runtime is {\it constant} with respect to the length $n=|u|$ of an input, see Theorem \ref{average-case} in Section \ref{orbits}. Sublinear-time algorithms represent a rather interesting class of algorithms: they can give some kind of an answer without reading the whole input but only a small part thereof. We refer to \cite{sublinear} for more details; here we just say that it is very rare to have a  decision algorithm that would give an exhaustive answer in sublinear time, let alone in constant time. Typically, it is either ``yes" or ``no" but not both. One well-known exception is deciding divisibility of a decimal integer by 2, 5, or 10: this is done by reading just the last digit. This is a rare instance where both ``yes" and ``no" answers can be given in sublinear (in fact, in constant) time. Less trivial examples, namely deciding subgroup membership in some special subgroups of $SL_2(\Z)$, were given in \cite{SL2} based on ideas from \cite{Sanov}.

\subsection{Las Vegas algorithms} \label{Vegas}

Our general method is described in this subsection.
We call an algorithm $\mathcal{A}$ {\it honest} if it always terminates in finite time and gives a correct result.

Then, a {\it Las Vegas algorithm} is a randomized algorithm that never gives an incorrect result; that is, it either  produces the correct result or it informs about the failure to obtain any result.

In contrast, a {\it Monte Carlo algorithm} is a randomized algorithm whose output may be incorrect with some (typically small) probability.

Las Vegas algorithms are more useful for our purposes because they can improve the time complexity of honest, ``hard-working", algorithms that always give a correct answer but are slow. Specifically, by running a fast
Las Vegas algorithm and a slow ``honest" algorithm in parallel, one often gets an algorithm that always terminates with a correct answer and whose average-case complexity is somewhere in between. This idea was used in \cite{average} where it was shown, in particular, that if a group $G$ has the word problem solvable in subexponential time and if $G$ has a non-amenable factor group where the word problem is solvable in a complexity class $\mathcal{C}$, then there is an honest algorithm that solves the word problem in $G$ with average-case complexity in $\mathcal{C}$.
\medskip

The arrangement of the paper is as follows. In Section \ref{orbits}, we show that primitivity of a cyclically reduced element $u \in F$ can be decided, on average, in constant time with respect to the length $|u|$ of the input word $u$. If one uses the ``Deque" (double-ended queue) model of computing \cite{deque}, then the ``cyclically reduced" condition on the input can be dropped.  In Section \ref{general}, we address the average-case complexity (i.e., the expected runtime) of the Whitehead algorithm in general, i.e., for an arbitrary input $(u, v) \in F_r \times F_r$. We show that in $F_2$, the average-case complexity is linear in $\max(|u|, |v|)$. For $r>2$, we show that the same is true, conditioned on a plausible conjecture about automorphic orbits, which is, however, still open. In the concluding Section \ref{blocking}, we offer a couple of interesting open problems about possible subwords of cyclically reduced  words in automorphic orbits. These problems, if answered, could yield a different approach to evaluating the average-case complexity of the Whitehead algorithm.

\section{Recognizing primitive elements of a free group} \label{orbits}

Let $F_r$ be a free group with a free generating set $x_1, \ldots, x_r$ and let $w=w(x_1, \ldots, x_r)$.
Call an element $u \in F_r$ {\it primitive} if there is an automorphism of $F_r$ that takes $x_1$ to $u$.

We will also need the definition of the Whitehead graph of an element $w \in F_r$. The Whitehead graph $Wh(w)$ of $w$ has $2r$  vertices that correspond to $x_1, \ldots, x_r, x_1^{-1}, \ldots, x_r^{-1}$. For  each occurrence of a subword  $x_i x_j$ in the  word $w \in F_r$, there
is an  edge in $Wh(w)$ that connects the vertex $x_i$ to the vertex
$x_j^{-1}$; ~if $w$ has a subword  $x_i x_j^{-1}$, then there is an
edge connecting $x_i$ to $x_j$, etc. ~There is one more edge (the
external edge): this is the edge that connects the vertex corresponding to the last
letter of $w$ to the vertex corresponding to the inverse of the
first letter.

It was observed by Whitehead himself (see \cite{Wh} or \cite{L-S}) that the Whitehead graph
of any cyclically reduced primitive element $w$ of length $>2$ has either an isolated edge or a cut vertex, i.e., a vertex that, having been removed from the graph together with all incident
edges, increases the number of connected components of the graph. A short and elementary proof of this result was recently given in \cite{Heusener}.

Now call a group word $w=w(x_1, \ldots, x_r)$ {\it primitivity-blocking} if it cannot be a subword of any cyclically reduced primitive element of $F_r$. For example, if the Whitehead graph of $w$ (without the external edge) is complete (i.e., any two vertices are connected by at least one edge), then $w$ is primitivity-blocking because in this case, if $w$ is a subword of $u$, then the Whitehead graph of $u$, too, is complete and therefore does not have a cut vertex or an isolated edge.

Examples of primitivity-blocking words are: $x_1^n x_2^n \cdots x_r^n x_1$ (for any $n \ge 2$), $[x_1, x_2][x_3, x_4]\cdots [x_{n-1}, x_{n}]x_1^{-1}$ (for an even $n$), etc. Here $[x, y]$ denotes $x^{-1}y^{-1}xy$. 

The usual algorithm for deciding whether or not a given element of $F_r$ is primitive is a special case of the general Whitehead algorithm that decides, given two elements $u, v \in F_r$,  whether or not $u$ can be taken to $v$ by an automorphism of $F_r$. If $v=x_1$, the worst-case complexity of the Whitehead algorithm is at most quadratic in $\max(|u|, |v|)=|u|$.

The generic-case complexity of the Whitehead algorithm was shown to be linear in \cite{White}.
Here we are going to address the average-case complexity of the Whitehead algorithm run in parallel with a ``fast check" algorithm, when applied to recognizing primitive elements of $F_r$.

A ``fast check" algorithm  $\mathcal{T}$ to test primitivity of an input (cyclically reduced) word $u$ would be as follows. Let $n$ be the length of $u$.  The algorithm  $\mathcal{T}$ would read the initial segments of $u$ of length $k$, $k=1, 2, \ldots,$ adding one letter at a time, and build the Whitehead graph of this segment, excluding the external edge. Then the algorithm would check if this graph is complete. If it is complete, then in particular it does not have a cut vertex or an isolated edge, so the input element is not primitive.

Note that the Whitehead graph always has $2r$ vertices, so checking the property of such a graph to be complete takes constant time with respect to the length of $u$, although reading a segment of $u$ of length $k$ takes time $O(k)$. If the Whitehead graph of $u$ is complete, the algorithm returns ``$u$ is not primitive". If it is not, the algorithm just stops.

Denote the ``usual" Whitehead algorithm by $\mathcal{W}$.
Now we are going to run the algorithms $\mathcal{T}$ and  $\mathcal{W}$ in parallel; denote the composite algorithm by $\mathcal{A}$. Then we have:

\begin{theorem}\label{average-case}
Suppose possible inputs of the above algorithm $\mathcal{A}$ are cyclically reduced words that are selected uniformly at random from the set of cyclically reduced words of length $n$.  Then the average-case time complexity (a.k.a expected runtime) of the algorithm $\mathcal{A}$, working on a classical Turing machine, is $O(1)$, a constant that depends on $r$ but not on $n$.
If one uses the ``Deque" (double-ended queue) model of computing \cite{deque} instead of a classical Turing machine, then the ``cyclically reduced" condition on the input can be dropped. 
\end{theorem}

\begin{proof}

Suppose first that the input word $u$ is cyclically reduced. 

\medskip

\noindent {\bf 1.} First we address complexity of the algorithm $\mathcal{T}$. Here we use
a result of \cite{languages} saying that the number of (freely reduced) words of length $L$ with (any number of) forbidden subwords  grows exponentially slower than the number of all freely reduced words of length $L$.

In our situation, if the Whitehead graph of a word $w$ is not complete, that means $w$ does not have at least one $x_i^{\pm 1}x_j^{\pm 1}$ as a subword. Thus, if the Whitehead graph of any initial segment of the input word $u$ is not complete, we have at least one forbidden subword. Therefore, the probability that the Whitehead graph of the initial segment of length $k$ of the word $u$ is not complete is $O(s^k)$ for some $s, ~0<s<1$. Thus, the average time complexity of the algorithm $\mathcal{T}$ is

\begin{equation}\label{T}
\sum_{k=1}^n k\cdot s^k,
\end{equation}

\noindent which is bounded by a constant.

In this case, the algorithm $\mathcal{T}$ terminates in time $O(f(n))$  and tells us that $u$ is not primitive. The expected runtime of the algorithm $\mathcal{A}$ in this case equals the expected runtime of the algorithm $\mathcal{T}$ and is therefore $O(f(n))$.

\medskip

\noindent {\bf 2.} Now suppose that the Whitehead graph of the input word $u$ of length $n$ is not complete, so that we have to employ the Whitehead algorithm $\mathcal{W}$. The probability of this to happen is $O(s^n)$ for some $s, ~0<s<1$, as was mentioned before.

Then, the worst-case complexity of the Whitehead algorithm for detecting primitivity of the input word is known to have time complexity $O(n^2)$.

Thus, the average-case complexity of the composite algorithm $\mathcal{A}$ is

\begin{equation}\label{A}
\sum_{k=1}^n k\cdot s^k + O(n^2) \cdot O(s^n),
\end{equation}

\noindent which is bounded by a constant.

\medskip

\noindent {\bf 3.} Now suppose the input word $u$ is not cyclically reduced. Then we are going to cyclically reduce it.
This cannot be done in constant (or even sublinear) time on a classical Turing machine, so here we are going to use the ``Deque" (double-ended queue) model of computing \cite{deque}. It allows one to move between the first and last letter of a word in constant time.

First, recall that the number of freely reduced words of length $n$ in $F_r$ is $2r(2r-1)^{n-1}$. We also note, in passing, that $u$ is primitive if and only if any conjugate of $u$ is primitive.

The following algorithm, that we denote by $\mathcal{B}$, will cyclically reduce $u$ on average in
constant time with respect to $n=|u|$.

This algorithm will compare the first letter of $u$, call it $a$, to the last letter, call it $z$. If $z \ne a^{-1}$, the algorithm stops right away. If $z = a^{-1}$, the first and last letters are deleted, and the algorithm now works with this new word.

The probability of $z = a^{-1}$ is $\frac{1}{2r}$ for any freely reduced word whose letters were selected uniformly at random from the set  $\{x_1, \ldots, x_r, x_1^{-1}, \ldots, x_r^{-1}\}$. At the next step of the algorithm, however, the letter immediately following $a$ cannot be equal to $a^{-1}$ if we assume that the input is a freely reduced word, so at the next steps (if any) of the algorithm $\mathcal{B}$ the probability of the last letter being equal to the inverse of the first letter will be $\frac{1}{2r-1}$. Then the expected runtime of the algorithm $\mathcal{B}$ on an input word of length $n$ is:

$$\sum_{k=1}^{\frac{n}{2}} \frac{1}{2r} (\frac{1}{2r-1})^{k-1} \cdot k < \sum_{k=1}^\infty (\frac{1}{2r-1})^k \cdot k.$$

\noindent The infinite sum on the right is known to be equal to $\frac{2r-1}{(2r-2)^2}$; in particular, it is constant with respect to $n$.

\end{proof}



\section{The general Whitehead algorithm} \label{general}

Recall that $F_r$ denotes a free group of rank $r$. We start by recalling briefly how the classical Whitehead algorithm works. There is a finite set $E$ of {\it elementary Whitehead automorphisms}; the cardinality of $E$ depends only on $r$. The Whitehead algorithm is in two parts. It takes two freely reduced words $u, v \in F_r$ as an input. Then, by applying automorphisms from the set $E$, it reduces both words to words of minimum   lengths in the corresponding automorphic orbits. This is a greedy algorithm since it was proved by Whitehead that applying automorphisms from $E$ can be arranged so that the length of a word is reduced at every step (the famous ``peak reduction"). The worst-case complexity of this part of the Whitehead algorithm is quadratic with respect to $\max(|u|, |v|)$.

If the lengths of the reduced words $\bar u, \bar v$ are different, then $u$ and $v$ are not in the same automorphic orbit. If the lengths are equal, the second part of the Whitehead algorithm is applied.  The procedure outlined in the original paper by Whitehead suggested this part of the algorithm to be of superexponential time with respect to $|\bar u|=|\bar v|$. However, a standard trick in graph theory (see \cite{MS}) shows that there is an algorithm of at most exponential time. In fact, it was shown in \cite{MS} (Proposition 3.1) that if $N$ is the number of automorphic images of $u \in F_r$ that have the same length as $u$ does, then given an element $v$ of length $|u|$, one can decide, in time linear in $N$, whether or not $v$ is an automorphic image of $u$. Thus, everything boils down to computing (or estimating) the number of automorphic images of $\bar u$ that have the same length as $\bar u$ does. This is still an open problem, see below.

However, if $r=2$, the worst-case complexity of the Whitehead algorithm is known to be polynomial-time in $\max(|u|, |v|)$ \cite{MS} (in fact, quadratic-time \cite{Khan}), and this, combined with results in \cite{White}, makes it easy to establish the following

\begin{proposition} \label{rank2}
The average-case time complexity of the Whitehead algorithm in $F_2$ is linear.
\end{proposition}

Proposition \ref{rank2} will follow from Proposition \ref{conditional} below.
For arbitrary $r>2$, the claim of Proposition \ref{rank2} is also true conditioned on the validity of  Conjecture 1 below.

\begin{conjecture}\label{conj1}
Let the length of $w \in F_r$ be irreducible by any elementary Whitehead automorphism (in particular, $w$ is cyclically reduced). Then the number of elements in the automorphic orbit of $w$ that have the same length as $w$ does is subexponential in $|w|$.
\end{conjecture}

Given a $w \in F_r$ that is shortest in its  orbit, denote by $\Gamma(w)$ the graph whose vertices are automorphic images of $w$, and two vertices are connected by an edge if and only if one of the corresponding elements of $F_r$ can be obtained from the other by applying a single elementary Whitehead automorphism. Then, denote by $\overline{\Gamma(w)}$ the subgraph of $\Gamma(w)$ that only includes vertices of $\Gamma(w)$ corresponding to words of the same length as $w$. Note that the graph $\overline{\Gamma(w)}$ is connected by a result of Whitehead \cite{Wh}.

Estimating the number of vertices in $\overline{\Gamma(w)}$ is not easy.
D. Lee showed \cite{Lee} that this number  is ``almost always" polynomial in $|w|$. Possible exceptions are words $w$ such that for some letters $x_i, x_j$, the number of occurrences of $x_i$ and  $x_i^{-1}$ (combined) equals the number of occurrences of $x_j$ and  $x_j^{-1}$ (combined). The set of words with this property is negligible in $F_r$, but this alone is not enough to get what we want; one needs a quantitative estimate of the growth rate of an arbitrary automorphic orbit. The automorphic orbit generated by $x_1$ (i.e., the set of all primitive elements of $F_r$) appears to be the ``thickest", i.e., it has the largest growth rate, known to be equal to $2r-3$, see \cite{Puder}. Other automorphic orbits appear to be ``much thinner"; specifically, we believe the following is true:

\begin{conjecture}\label{conj2} (cf. \cite{Puder})
Let $w \in F_r$ be a non-primitive element. Let $CA(w)$ denote the set of cyclically reduced automorphic images of $w$. Then the growth rate of $CA(w)$ is less than $\sqrt{2r-1}$.
\end{conjecture}

Conjecture \ref{conj2} is obviously weaker than Conjecture \ref{conj1} is. We also mention here two related conjectures from
\cite{Puder}:

\medskip

\noindent {\bf PW1.} (See also \cite{Problems}, Problem (F41)). Let $w \in F_r$ be a non-primitive element. Then the growth rate of the orbit $Orb(w)=\{\varphi(w), ~\varphi \in Aut(F_r)\}$ is at most $\sqrt{2r-1}$. Note that $\sqrt{2r-1}$ is the growth rate of the set of all conjugates of $w$. Thus, this conjecture basically says that ``most" automorphic images of a non-primitive element are its conjugates.
\medskip

\noindent {\bf PW2.} Let $w \in F_r$ and let $\mu(w)$ denote the minimum (positive) number of occurrences of any $x_i$ in $w$.
Then the growth rate of the set of {\it cyclically reduced} words in $Orb(w)$ is $\sqrt[\mu(w)]{2r-3}$.
\medskip

Proving Conjecture 1 would be sufficient to generalize Proposition \ref{rank2} to arbitrary $r\ge 2$:

\begin{proposition}\label{conditional} Assume that Conjecture \ref{conj1} holds in $F_r$. Then there is an algorithm $\mathcal{A}$ that, on input $(u, v)$, where $u, v \in F_r$, decides whether or not  $v=\varphi(u)$ for some $\varphi \in Aut(F_r)$, and the expected runtime of the algorithm $\mathcal{A}$ is at most linear in $\max(|u|, |v|)$.
\end{proposition}

\begin{proof}
First we note that Conjectures 1, 2 in this section hold in $F_2$ because in $F_2$, the number of cyclically reduced elements of length $n$ in the same automorphic orbit is bounded by a quadratic function of $n$, see \cite{Khan} or \cite[Section 4.2.3]{KhanThesis}.

Now let $r>2$ be arbitrary, and let $(u, v) \in F_r \times F_r$ be an input. We can assume that both  $u$ and $v$ are cyclically reduced; if not, then cyclically reducing them takes linear time in $\max(|u|, |v|)$ anyway.

Our algorithm $\mathcal{A}$ will consist of three algorithms, $\mathcal{C}$, $\mathcal{R}$, and  $\mathcal{S}$, run in succession. Before we describe these algorithms and compute the expected runtime of the whole algorithm $\mathcal{A}$, we need to make some assumptions on how inputs $(u, v) \in F_r \times F_r$ of complexity $n$ are sampled. Recall that we define complexity of $(u, v)$ as $\max(|u|, |v|)$. Accordingly, we first select one of the two elements, say, $u$ uniformly at random from all elements of $F_r$ having length $n$. After that, we select the other element $v$ uniformly at random from all elements of $F_r$ having length $\le n$.
\medskip

\noindent {\bf 1.} First we employ the algorithm that we denote by $\mathcal{C}$. Following \cite{White}, denote by $SM$ the set of ``strictly minimal" words. These are cyclically reduced words $w$ with the following property: any automorphism of $F_r$ except permutations on the set of generators of $F_r$ and their inverses and those inducing cyclic permutations of $w$, strictly increase the length of $w$.

By \cite[Theorem A]{White}, the set $SM$ is exponentially generic. By the same theorem,
there is a linear time algorithm that tells whether or not a given cyclically reduced word of length $w$ is in the set $SM$. Apply this algorithm to both $u$ and $v$. Suppose first that both $u$ and $v$ are in $SM$.

If $u$ and $v$ are not of the same length, then the algorithm $\mathcal{C}$ concludes that $u$ and $v$ are not in the same automorphic orbit.

If $u$ and $v$ are of the same length, then we have to check if there is a composition of a cyclic permutation of $u$ and a permutation on the set $X^{\pm 1}$ of generators of $F_r$ and their inverses that takes $u$ to $v$. It may seem that this would take quadratic time in $n=|u|=|v|$ because there are $n$ cyclic permutations of $u$, and for each we have to apply permutations on $X^{\pm 1}$ to $u$, which takes linear time in $n$.

However, we first observe that the word $u^2$ contains all cyclic permutations of $u$ as subwords, and any subword of $u^2$ of length $n$ is a cyclic permutation of $u$.
Then, a linear time string-pattern-matching recognition algorithm of Knuth–Morris–Pratt \cite{Knuth} can be applied to check if $v$ is a subword of $u^2$ and therefore is a cyclic permutation of $u$.

Thus, if $u$ and $v$ are both in the set $SM$, then we can decide in time linear in $n=\max(|u|, |v|)$ whether or not they are in the same automorphic orbit.

\medskip

\noindent {\bf 2.} Now suppose at least one of $u$ and $v$ is not in the set $SM$. Let $n=\max(|u|,|v|)=|u|$. If $u$ is in $SM$ and $v$ is not, then $u$ and $v$ are not in the same automorphic orbit.

Now suppose $u$ is not in $SM$. The probability of this to happen is $O(\exp(-n))$ by \cite[Theorem A]{White}.
We then apply the algorithm $\mathcal{R}$ to $u$. This algorithm checks if $|u|$ can be reduced by applying elementary Whitehead automorphisms (there is a constant number of those) to $u$, and if so, the algorithm $\mathcal{C}$ reduces $|u|$ until it cannot be reduced any further. This takes time
$O(n^2)$.

Now suppose both $u$ and $v$ are not in $SM$. Again, the probability of this to happen is $O(\exp(-n))$. If both $|u|$ and $|v|$ can be reduced by elementary Whitehead automorphisms, we
apply the algorithm $\mathcal{R}$ to $u$ and then to $v$ to reduce their length until it cannot be reduced any further. This takes time $O(n^2)$. If the reduced elements $\bar{u}$ and $\bar{v}$ are not of the same length, then the algorithm $\mathcal{R}$ tells that $u$ and $v$  are not in the same automorphic orbit.
If they are of the same length, we proceed to the remaining case.
\medskip


\noindent {\bf 3.} The remaining case is where the length $m=|\bar{u}|=|\bar{v}|$ cannot be reduced by elementary Whitehead automorphisms. Recall that we are currently in the case where both $u$ and $v$ are not in the set $SM$, which happens with probability $O(\exp(-n))$, where $n\ge m$ is the maximum of the lengths of the original input words $u$ and $v$.

Now we apply the algorithm $\mathcal{S}$, the most mysterious one (from the complexity point of view). It takes  $\bar{u}$ as input, builds the graph $\overline{\Gamma(\bar{u})}$ and checks whether or not $\bar{v}$ labels one of the vertices of $\overline{\Gamma(\bar{u})}$. Recall that  $|\bar{u}|=m$.

To complete the proof of Proposition \ref{conditional} (and therefore also of Proposition \ref{rank2}), we need to estimate the runtime of the algorithm $\mathcal{S}$, under the assumption that Conjecture \ref{conj1} holds in $F_r$. Once again, the input for the algorithm $\mathcal{S}$ is a pair of words $\bar{u}$, $\bar{v}$ of the same length $m, ~ 1 < m \le n,$ irreducible by any elementary Whitehead automorphism. 

One can then start building a spanning tree for $\overline{\Gamma(\bar{u})}$ using the breadth-first search method, with the vertex $\bar{u}$ as the root. Since the degree of every vertex of $\overline{\Gamma(\bar{u})}$ is bounded by a function of $r$ only and therefore is constant with respect to $m=|\bar{u}|$, the breadth-first search should terminate in time $O(|\overline{\Gamma(\bar{u})}|)$, where $|\overline{\Gamma(\bar{u})}|$ is the number of vertices in the graph $\overline{\Gamma(\bar{u})}$. Therefore, the expected runtime of the algorithm $\mathcal{S}$ is

$$\displaystyle{O(\exp(-n)) \cdot |\overline{\Gamma(\bar{u})}|)},$$

\noindent which is sublinear in $n=|u|$, assuming that Conjecture \ref{conj1} holds in $F_r$, i.e., assuming that the number of vertices in the graph $\overline{\Gamma(\bar{u})}$ is subexponential in $m=|\bar{u}|$, and therefore also in $n=|u|$.

Thus, summing up the expected runtime of the algorithms $\mathcal{C}$, $\mathcal{R}$, and $\mathcal{S}$, we have:

$$\displaystyle{O(n) + O(\exp(-n)) \cdot O(n^2) + O(\exp(-n)) \cdot |\overline{\Gamma(\bar{u})}|)} = O(n).$$

This completes the proof.

\end{proof}


\section{Orbit-blocking words}\label{blocking}

Another possible approach to generalizing Proposition \ref{rank2} to arbitrary $r>2$ is probably harder but is of independent interest. Let us generalize the definition of primitivity-blocking words to other orbit-blocking words as follows.
Let $u \in F_r$. Consider the orbit $Orb(u)=\{\varphi(u), ~\varphi \in Aut(F_r)\}$. Call $w \in F_r$ an
$Orb(u)$-blocking word if it cannot be a subword of any cyclically reduced $v \in Orb(u)$. We ask:

\begin{problem}\label{exist} (see also \cite{Problems}, Problem (F40)).
Is it true that there are $Orb(u)$-blocking words for any $u \in F_r$ ?
\end{problem}

\begin{problem}\label{output}
Is there an algorithm that, on input $u \in F_r$, would output at least one particular $Orb(u)$-blocking word, assuming there is one?
\end{problem}

A good starting point might be to find an $Orb(u)$-blocking word for $u=[x_1, x_2]$. This is easy to do if $r=2$ since, by a classical result of Nielsen (see e.g. \cite{L-S}), any cyclically reduced $v \in Orb([x_1, x_2])$ in this case is either $[x_1, x_2]$ or $[x_2, x_1]$. Thus, almost every word in this case is $Orb(u)$-blocking. This is only true if $r=2$ though; if $r>2$, the situation is quite different.

\vskip .5cm

\noindent {\it Acknowledgement.} I am grateful to Ilya Kapovich for useful comments and discussions and to Sasha Ushakov for suggesting how to strengthen the original version of Theorem 1 (from ``sublinear" to ``constant" average-case complexity) and how to simplify the original proof of Proposition 2.

\end{document}